\documentclass[reqno]{amsart}
\usepackage[utf8]{inputenc}

\usepackage{amssymb}
\usepackage{graphicx}
\usepackage{latexsym}
\usepackage{stmaryrd}
\usepackage{enumitem}
\usepackage[bookmarks]{hyperref}
\usepackage{multirow}
\usepackage{xcolor}
\usepackage{relsize}
\usepackage{comment}
\usepackage{xifthen}
\usepackage{mathrsfs}
\usepackage{svg}
\usepackage{mathtools}
\colorlet{darkishRed}{red!80!black}
\colorlet{darkishBlue}{blue!60!black}
\colorlet{darkishGreen}{green!60!black}

\usepackage{hyperref}
\hypersetup{
    unicode,
    colorlinks,
    linkcolor={red!60!black},
    citecolor={green!60!black},
    urlcolor={blue!60!black}
}

\usepackage{geometry}
\renewcommand{\subset}{\subseteq}

\newcommand{\Abs}[1]{\partial_{\Omega} {#1}}

\newcommand{ \N } { \mathbb{N} }

\makeatletter

\def\calCommandfactory#1{%
   \expandafter\def\csname c#1\endcsname{\mathcal{#1}}}
\def\frakCommandfactory#1{%
   \expandafter\def\csname frak#1\endcsname{\mathfrak{#1}}}

\newcounter{ctr}
\loop
  \stepcounter{ctr}
  \edef\X{\@Alph\c@ctr}
  \expandafter\calCommandfactory\X
  \expandafter\frakCommandfactory\X
  \edef\Y{\@alph\c@ctr}
  \expandafter\frakCommandfactory\Y
\ifnum\thectr<26
\repeat

\renewcommand{\cC}{\mathscr{C}}

\newtheorem{theorem}{Theorem}[section]

\newtheorem{mainresult}{Theorem}
\newtheorem{mainproblem}{Problem}

\newtheorem{corollary}[theorem]{Corollary}
\newtheorem{lemma}[theorem]{Lemma}

\newtheorem{innercustomthm}{Theorem}

\newenvironment{customthm}[1]
  {\renewcommand\theinnercustomthm{#1}\innercustomthm}
  {\endinnercustomthm}

\theoremstyle{definition}

\theoremstyle{remark}

\newtheorem*{ack}{Acknowledgements}

\setenumerate{label={\normalfont (\roman*)},itemsep=0pt}

\title{Approximating infinite graphs by normal trees}

\author{Jan Kurkofka}
\address{University of Hamburg, Department of Mathematics, Bundesstraße 55 (Geomatikum), 20146 Hamburg, Germany}
\email{\{jan.kurkofka, ruben.melcher, max.pitz\}@uni-hamburg.de}

\author{Ruben Melcher}
\author{Max Pitz}



\begin{document}
\begin{abstract}
We show that every connected graph can be approximated by a normal tree, up to some arbitrarily small error phrased in terms of neighbourhoods around its ends.
The existence of such approximate normal trees has consequences of both combinatorial and topological nature.

On the combinatorial side, we show that a graph has a normal spanning tree as soon as it has normal spanning trees locally at each end;  i.e., the only obstruction for a graph to having a normal spanning tree is an end for which none of its neighbourhoods has a normal spanning tree.

On the topological side, we show that the end space $\Omega(G)$, as well as the spaces $|G| = G \cup \Omega(G)$ naturally associated with a graph $G$, are always paracompact. This gives unified and short proofs for a number of results by Diestel, Sprüssel and Polat, and answers an open question about metrizability of end spaces by Polat. 
\end{abstract}
\maketitle

\section{Introduction}

\noindent A rooted tree $T$ contained in a graph $G$ is \emph{normal} in $G$ if the endvertices of every \mbox{$T$-path} in $G$ are comparable in the tree-order of~$T$. 
(In finite graphs, normal spanning trees are their depth-first search trees; see~\cite{diestel2015book} for precise definitions.)
Normal spanning trees are perhaps the most useful structural tool in infinite graph theory. 
Their importance arises from the fact that they capture the separation properties of the graph they span, and so in many situations it suffices to deal with the much simpler tree structure instead of the whole graph. 
For example, the end space of $G$ coincides, even topologically, with the end space of any normal spanning tree of $G$. However, not every connected graph has a normal spanning tree, and the structure of graphs without normal spanning trees is still not completely understood~\cite{BowlerGeschkePitzNST,DiestelLeaderNST}.

In order to harness and transfer the power of normal spanning trees to arbitrary connected graphs~$G$, one might try to find an `approximate normal spanning tree': a normal tree in $G$ which spans the graph up to some arbitrarily small given error term. 
To formalize this idea, recall that, as usual, a \emph{neighbourhood} of an end is the component of $G-X$ which contains a tail of every ray of that end, for some (arbitrarily large) finite 
set of vertices $X\subseteq V(G)$. 
We say that a graph $G$ can be \emph{approximated by normal trees} if for every selection of arbitrarily small neighbourhoods around its ends there is a normal tree $T\subset G$ such that every component of $G-T$ is included in one of the selected neighbourhoods and every end of $G$ has some neighbourhood in $G$ that avoids $T$.

Our approximation result for normal trees in infinite graphs then reads as follows:

\begin{customthm}{\ref{thm_paracompact}}
Every connected graph can be approximated by normal trees.
\end{customthm}

\noindent Note that the normal trees provided by our theorem will always be rayless, see also Section~\ref{sec_para}.

We indicate the potential of Theorem~\ref{thm_paracompact} by a number of applications. 
Our first two applications are of combinatorial nature: we exhibit in Section~\ref{sec_cor} two new existence results for normal spanning trees that Theorem~\ref{thm_paracompactintro} implies.
One of these, Theorem~\ref{thm_local NST gives global NST}, says that if every end of a connected graph $G$ has a neighbourhood which has a normal spanning tree then $G$ itself has a normal spanning tree. 

Interestingly, Theorem \ref{thm_paracompact} may not only be read as a structural result for connected graphs: it also implies and extends a number of previously hard results about topological properties of end spaces \cite{diestel1992end, diestel2006end, diestel2003graph, polat1996ends, polat1996ends2, sprussel2008end}.
Denote by $\Omega(G)$ the end space of a graph $G$, and by $\vert G\vert$ the space on $G\cup\Omega(G)$  naturally associated with the graph $G$ and its ends; see the next section for precise definitions.
When $G$ is locally finite and connected, then $\Omega(G)$ is compact, and $\vert G\vert$ is the well-known Freudenthal compactification of $G$.
For arbitrary $G$, the spaces $\Omega(G)$ and $\vert G\vert$ are usually non-compact and far from being completely understood.

Polat has shown that $\Omega(G)$ is ultrametrizable if and only if $G$ contains a topologically end-faithful normal tree \cite[Theorem~5.13]{polat1996ends}, and has proved as a crucial auxiliary step that end spaces are always collectionwise normal \cite[Lemma~4.14]{polat1996ends}. Changing focus from $\Omega(G)$ to $|G|$, Sprüssel has shown that $|G|$ is normal \cite{sprussel2008end}, and Diestel has characterised when $|G|$ is metrizable or compact \cite{diestel2006end} in terms of certain normal spanning trees in $G$.
Our combinatorial Theorem~\ref{thm_paracompactintro} provides, in just a few lines, new and unified proofs for all these results. Additionally, Theorem~\ref{thm_paracompactintro} shows that metrizable end spaces are always ultrametrizable (Theorem~\ref{cor_char. OmegaG metr.}), answering an open question by Polat. 

Finally, Theorem~\ref{thm_paracompactintro} brings new progress to an old problem of Diestel, which asks for a topological charactersation of all end spaces {\cite[Problem~5.1]{diestel1992end}}. 
Indeed, note that Theorem~\ref{thm_paracompactintro} translates to the topological assertion that every open cover of an end space can be refined to an open partition cover, Corollary~\ref{cor_ultrapara}. 
This last property is known in the literature as ultra-paracompactness. It implies that all spaces $|G|$ are paracompact (Corollary~\ref{cor_modGpara}), and that all end spaces $\Omega(G)$ are even hereditarily ultra-paracompact (Corollary~\ref{cor_herpara}). 

This paper is organised as follows: The next section contains a recap on end spaces and other technical terms. Section~\ref{sec_para} contains the proof of our main result, and Section~\ref{sec_cor} derives the consequences outlined above. 
Section~\ref{sec_last} indicates a simple argument showing that subspaces of end spaces inherit their property of being ultra-paracompact. 

\begin{ack}
We are grateful to our referees for their comments and a  suggestion that simplified the proof of our main result.
\end{ack}

\quad \newline

\noindent

\section{End spaces of graphs: a reminder}
\label{sec_prelims}

For graph theoretic terms we follow the terminology in \cite{diestel2015book}, and in particular \cite[Chapter~8]{diestel2015book} for ends in graphs and the spaces $\Omega(G)$ and $|G|$. 
A $1$-way infinite path is called a \emph{ray} and the subrays of a ray are its \emph{tails}. Two rays in a graph $G = (V,E)$ are \emph{equivalent} if no finite set of vertices separates them; the corresponding equivalence classes of rays are the \emph{ends} of $G$. The set of ends of a graph $G$ is denoted by $\Omega = \Omega(G)$. If $X \subseteq V$ is finite and $\omega \in \Omega$, 
there is a unique component of $G-X$ that contains a tail of every ray in $\omega$, which we denote by $C(X,\omega)$. 
If $C$ is any component of $G-X$, we write $\Omega(X,C)$ for the set of ends $\omega$ of $G$ with $C(X,\omega) = C$, and abbreviate $\Omega(X,\omega) := \Omega(X,C(X,\omega))$. 
Finally, if $\cC$ is any collection of components of $G-X$, we write $\Omega(X,\cC) := \bigcup\,\{\,\Omega(X,C) \colon C \in \cC\,\}$.

The collection of all sets $\Omega(X,C)$ with $X\subset V$ finite and $C$ a component of $G-X$ form a basis for a topology on $\Omega$.
This topology is Hausdorff, and it is \emph{zero-dimensional} in that it has a basis consisting of closed-and-open sets. Note that when considering end spaces $\Omega(G)$, we may always assume that $G$ is connected; adding one new vertex and choosing a neighbour for it in each component does not affect the end space.

We now describe two common ways to extend this topology on $\Omega(G)$ to a topology on $|G| = G \cup \Omega(G)$, the graph $G$ together with its ends. 
The first topology, called \textsc{Top}, has a basis formed by all open sets of $G$ considered as a $1$-complex, together with basic open neighbourhoods for ends of the form
\begin{align*}
    \hat{C}_*(X,\omega) := C(X,\omega) \cup \Omega(X,\omega) \cup \mathring{E}_*(X, C(X,\omega)),
\end{align*}
where $\mathring{E}_*(X, C(X,\omega))$ denotes any union of half-open intervals of all the edges from the edge cut $E(X, C(X,\omega))$ with endpoint in $C(X,\omega)$.

As the $1$-complex topology on $G$ is not first-countable at vertices of infinite degree, it is sometimes useful to consider a metric topology on $G$ instead: 
The second topology commonly considered, called \textsc{MTop}, has a basis formed by all open sets of $G$ considered as a metric length-space (i.e.\ every edge together with its endvertices forms a unit interval of length~$1$, and the distance between two points of the graph is the length of a shortest arc in $G$ between them), together with basic open neighbourhoods for ends of the form
\begin{align*}
    \hat{C}_\varepsilon(X,\omega) := C(X,\omega) \cup \Omega(X,\omega) \cup \mathring{E}_\varepsilon(X, C(X,\omega)),
\end{align*}
where $\mathring{E}_\varepsilon(X, C(X,\omega))$ denotes the open ball around $C(X,\omega)$ in $G$ of radius $\varepsilon$. Note that both topologies \textsc{Top} and \textsc{MTop} induce the same subspace topology on $\hat{V}(G) := V(G) \cup \Omega(G)$ and $\Omega(G)$, the last of which coincides with the topology on $\Omega(G)$ described above. Polat observed that $\hat{V}(G)$ is homeomorphic with $\Omega(G^+)$, where $G^+$ denotes the graph obtained from $G$ by gluing a new ray $R_v$ onto each vertex $v$ of $G$ so that $R_v$ meets $G$ precisely in its first vertex $v$ and $R_v$ is distinct from all other $R_{v'}$, cf.~\cite[\S4.16]{polat1996ends}.

A \emph{direction} on $G$ is a function $d$ that assigns to every finite $X \subseteq V$ one of the components of $G-X$ so that $d(X) \supseteq d(X')$ whenever $X \subseteq X'$. For every end $\omega$, the map $X \mapsto C(X,\omega)$ is easily seen to be a direction. Conversely, every direction is defined by an end in this way:

\begin{theorem}[Diestel \& K\"uhn {\cite{diestel2003graph}}]
\label{thm_direction}
For every direction $d$ on a graph $G$ there is an end $\omega$ such that $d(X) = C(X, \omega)$ for every finite $X \subseteq V(G)$.
\end{theorem}

The \emph{tree-order} of a rooted tree $(T,r)$ is defined by setting $u \leq v$ if $u$ lies on the unique path $rTv$ from $r$ to $v$ in $T$. Given $n \in \N$, the \emph{$n$th level} of $T$ is the set of vertices at distance $n$ from $r$ in $T$. The \emph{down-closure} of a vertex $v$ is the set $\lceil v \rceil := \{\,u \colon u \leq v\,\}$; its \emph{up-closure} is the set $\lfloor v \rfloor := \{\,w \colon v \leq w\,\}$. The down-closure of $v$ is always a finite chain, the vertex set of the path $rTv$. A ray $R \subset T$ starting at the root is called a \emph{normal ray} of $T$.

A rooted tree $T$ contained in a graph $G$ is \emph{normal} in~$G$ if the endvertices of every \mbox{$T$-path} in $G$ are comparable in the tree-order of~$T$.
Here, for a given graph~$H$, a path $P$ is said to be an $H$-\emph{path} if $P$ is non-trivial and meets $H$ exactly in its endvertices.
We remark that for a normal tree $T\subset G$ the neighbourhood $N(D)$ of every component $D$ of $G-T$ forms a chain in $T$. 
A set $U$ of vertices is \emph{dispersed} in $G$ if for every end $\omega$ there is a finite $X \subset V$ with $C(X,\omega) \cap U = \emptyset$, or equivalently, if $U$ is a closed subset of $|G|$ (in either \textsc{Top} or \textsc{MTop}).

\begin{theorem}[Jung {\cite{jung1969wurzelbaume}}]
\label{thm_jung}
A vertex set in a connected graph is dispersed if and only if there is a rayless normal tree including it. 
Moreover, every rayless normal tree in a connected graph can be extended to a rayless normal tree that includes an arbitrary pre-specified dispersed vertex set of the graph.
As a consequence, a connected graph has a normal spanning tree if and only if its vertex set is a countable union of dispersed~sets.
\end{theorem}

If $H$ is a subgraph of $G$, then rays equivalent in $H$ remain equivalent in $G$; in other words, every end of $H$ can be interpreted as a subset of an end of $G$, so the natural inclusion map $\iota \colon \Omega(H) \to \Omega(G)$ is well-defined. A subgraph $H \subset G$ is \emph{end-faithful} if this inclusion map $\iota$ is a bijection. The terms \emph{end-injective} and \emph{end-surjective} are defined accordingly. Normal trees are always end-injective; hence, normal trees are end-faithful as soon as they are end-surjective. Given a subgraph $H \subset G$, write $\Abs{H} \subset \Omega(G)$ for the set of ends $\omega$ of $G$ which satisfy $C(X,\omega) \cap H \neq \emptyset$ for all finite $X \subset V(G)$.

For topological notions we follow the terminology in \cite{EngelkingBook}. All spaces considered in this paper are Hausdorff, i.e.\ every two distinct points have disjoint open neighbourhoods. An \emph{ultrametric} space $(X,d)$ is a metric space in which the triangle inequality is strengthened to $d(x,z)\leq\max\left\{d(x,y),d(y,z)\right\}$. 
A~topological space $X$ is \emph{ultrametrizable} if there is an ultrametric $d$ on $X$ which induces the topology of $X$. A topological space is \emph{normal} if for any two disjoint closed sets $A_1,A_2$ there are disjoint open sets $U_1,U_2$ with $A_i \subset U_i$. A space is \emph{collectionwise normal} if for every \emph{discrete} family $\{\,A_s \colon s \in S\,\}$ of disjoint closed sets, i.e.\ a family such that $\bigcup \,\{\,A_s \colon s \in S'\,\}$ is closed for any $S' \subseteq S$, there is a collection $\{\,U_s \colon s \in S\,\}$ of disjoint open sets with $A_s\subset U_s$. 

A collection $\cA$ of sets is said to \emph{refine} another collection $\cB$ of sets if for every $A \in \cA$ there is $B \in \cB$ with $A \subseteq B$. A cover $\cV$ of a topological space $X$ is \emph{locally finite} if every point of $X$ has an open neighbourhood which meets only finitely many elements of~$\cV$. A topological space $X$ is \emph{paracompact} if for every open cover $\cU$ of $X$ there is a locally finite open cover $\cV$ refining $\cU$. 
All compact Hausdorff spaces and also all metric spaces are paracompact, which in turn are always normal and even collectionwise normal \cite[Chapter~5.1]{EngelkingBook}. A space is \emph{ultra-paracompact} if every open cover has a refinement by an open partition.

Lastly, ordinal numbers are identified with the set of all smaller ordinals, i.e.\ $\alpha = \{\,\beta\colon\beta < \alpha\,\}$ for all ordinals $\alpha$.

\section{Proof of the main result}
\label{sec_para}

This section is devoted to the proof of our main theorem, which we restate more formally:

\begin{mainresult}
\label{thm_paracompactFast}
\label{thm_paracompactintro}
\label{thm_paracompact}
For every collection $\cC=\{\,C(X_\omega,\omega) \colon \omega \in \Omega(G)\,\}$ in a connected graph $G$, there is a rayless normal tree $T$ in $G$ such that every component of $G-T$ is included in an element of $\cC$.
\end{mainresult}

As every rayless normal tree $T\subset G$ is dispersed in $G$ by Jung's Theorem~\ref{thm_jung}, this technical variant of our main result is clearly equivalent to the formulation presented in the introduction.

Let us briefly discuss two other possible notions of `approximating graphs by normal trees': First, Theorem~\ref{thm_paracompact} is significantly stronger than just requiring that (every component of) $G-T$ is included in the union $\bigcup \cC$ of the selected neighbourhoods; the latter assertion is easily seen to be equivalent to Jung's Theorem \ref{thm_jung}.
In the other direction, could one strengthen our notion of `approximating by normal trees' and demand a normal rayless tree $T$ such that for every end $\omega$ of $G$, the component of $G-T$ in which every ray of $\omega$ has a tail is included in $C(X_\omega, \omega)$? This notion, however, is too strong and such a $T$ may not exist: Consider the graph $G=K^+$ (see Section~\ref{sec_prelims}) for an uncountable clique $K$, and let $\cC$ be the collection of all the ray-components of $G-K$ (together with an arbitrary neighbourhood of the end of the clique $K$). Any normal tree for $G$ satisfying our stronger requirements would restrict to a normal spanning tree of $K$, an impossibility.

\begin{proof}[Proof of Theorem~\ref{thm_paracompactFast}]

Given a collection $\cC=\{\,C(X_\omega,\omega)\colon \omega\in\Omega(G)\,\}$ in a connected graph $G$, call a subgraph $H\subset G$ \emph{bounded} if there is an end $\omega\in\Omega(G)$ with $H\subset C(X_\omega,\omega)$, and \emph{unbounded} otherwise.

We construct a sequence of rayless normal trees $T_1 \subset T_2 \subset \ldots$ extending each other all with the same root $r$ as follows:
Let $T_1$ be the tree on a single vertex $r$ (for some arbitrarily chosen vertex $r\in G$) and suppose that $T_n$ has already been constructed.

We claim that for every unbounded component $D$ of $G - T_n$ there exists a finite separator $S_D \subset V(D)$ such that $D-S_D$ has either zero or at least two unbounded components. To see this, suppose the contrary; then the map $d$ sending each finite vertex set in $D$ to its unique unbounded component is a direction on $D$ and hence defines an end $\omega$ of $D$ by Theorem~\ref{thm_direction}. But $d(X_\omega\cap D)\subset C(X_\omega,\omega)$ is bounded, a contradiction.

Now for every such unbounded $D$ let $S_D$ be a finite separator of the first kind in $D$ if possible, and otherwise of the second kind. Note that the union of all finite vertex sets $S_D$ is dispersed in~$G$ because $T_n$ is a rayless normal tree.
Since $G$ is connected, we may use Jung's Theorem~\ref{thm_jung} to extend $T_n$ in an inclusion minimal way to a rayless normal tree $T_{n+1}\supseteq T_n$ with root $r$ that includes all the finite vertex sets~$S_D$. This completes the construction.  

Now consider the normal tree $T = \bigcup_{n \in \mathbb{N}} T_n$. 
We claim that $T$ is rayless. Indeed, suppose otherwise, that there is a normal ray $R$ in $T$ belonging to the end $\omega \in \Omega(G)$ say.

Then, for every $n \in \N$, the ray $R$ has a tail in an unbounded component $D_n$ of $G-T_n$, and all finite separators $S_{D_n}$ chosen for these components were of the second kind, since we never extended $T_n$ into a component that was already bounded. In particular $R$ meets each $S_{D_n}$ in at least one vertex, $s_n$ say. Now, fix for every $S_{D_n}$ an unbounded component $C_{n+1}$ of $D_n - S_{D_n} $ different from $D_{n+1}$. Every $C_{n+1}$ has a neighbour, say $u_n$, in $S_{D_n}$. Moreover, the paths $P_n=s_n T u_n$ connecting $s_n$ to $u_n$ in $T$ are pairwise disjoint, as each of them was constructed in the $n$th step.

From this, we obtain a contradiction as follows. Choose $n \in \N$ large enough so that $X_{\omega}$ avoids all of $C_{n+1}$, the path $P_n$ and the tail of $R$ that starts at the vertex $s_n$ (this is indeed possible since any two components $C_i$ and $C_j$ are disjoint). Now, $C_{n+1}$ is contained in $C(X_{\omega}, \omega)$, contradicting the fact that $C_{n+1}$ is unbounded. This shows that $\omega$ cannot exist, and hence that $T$ is rayless.

Finally, we claim that every component $D$ of $G - T = G - \bigcup_{n \in \N} T_n$ is bounded. Since $T$ is a normal tree, $N(D)$ is a chain in $T$, and since $T$ is rayless, $N(D)$ is finite. Hence, there is $m \in \N$ such that $N(D) \subset T_m$, i.e.\ $D$ is already a component of $G-T_m$.  The fact that we have not extended $T_m$ into $D$ means that $D$ is bounded.
\end{proof}

\begin{corollary}
\label{cor_ultrapara}
For every connected graph $G$ and every open cover $\mathcal{U}$ of its end space $\Omega(G)$ there is a rayless normal tree $T$ in $G$ such that the collection of components of $G - T$ induces an open partition of $\Omega(G)$ refining $\mathcal{U}$.
In  particular, all end spaces $\Omega(G)$ are ultra-paracompact.
\end{corollary}
\begin{proof}
Without loss of generality, the open cover $\mathcal{U} $ is of the form $\mathcal{U} =\{\,\Omega(X_\omega,\omega) \colon \omega \in \Omega(G)\,\}$. Theorem~\ref{thm_paracompactFast} applied to $\cC=\{\,C(X_\omega,\omega) \colon \omega \in \Omega(G)\,\}$ yields a rayless normal tree $T$ in $G$ such that every component of $G-T$ is included in an element of $\cC$. As every component of $G-T$ has a finite neighbourhood, it induces an open set of the end space. This gives the desired open partition of $\Omega(G)$ refining $\mathcal{U}$.
\end{proof}

\begin{corollary}
\label{cor_modGpara}
All spaces $|G|$ are paracompact in both \textsc{Top} and \textsc{MTop}. \end{corollary}
\begin{proof}
First, we consider $|G|$ with \textsc{MTop}.
To show that $|G|$ is paracompact, suppose that any open cover $\cU$ of $|G|$ consisting of basic open sets is given.
The cover elements come in two types: basic open sets of $G$, and basic open neighbourhoods of ends.
We write $\cU_\Omega= \{\,\hat{C}_{\varepsilon_i}(X_i,\omega_i) \colon i \in I\,\}$ for the collection consisting of the latter.
As $\cU_\Omega$ covers the end space of $G$, applying Theorem~\ref{thm_paracompact} to the collection $\cC:=\{\,C(X_i,\omega_i)\colon i\in I\,\}$ yields a rayless normal tree $T$ in $G$ such that $\{\,C(Y_j,\omega_j)\colon j\in J\,\}$, the collection of components of $G-T$ containing a ray, refines~$\cC$.
For every $j\in J$ we choose $\varepsilon_j:=\varepsilon_i$ for some $i\in I$ with $C(Y_j,\omega_j)\subset C(X_i,\omega_i)$, ensuring that the disjoint collection $\cV_\Omega:=\{\,\hat{C}_{\varepsilon_j}(Y_j,\omega_j)\colon j\in J\,\}$ refines~$\cU_\Omega$.

Next, consider the quotient space $H$ that is obtained from $|G|$ by collapsing every closed subset $C(Y_j,\omega_j)\cup\Omega(Y_j,\omega_j)$ with $j\in J$ to a single point.
As the open sets in $\cV_\Omega$ are disjoint, the quotient is well-defined and we may view $H$ as a rayless multi-graph endowed with \textsc{MTop}.
Now consider the open cover $\cU_H$ of $H$ that consists of the quotients of the elements of $\cV_\Omega$ on the one hand, and on the other hand, for every non-contraction point of $H$ a choice of one basic open neighbourhood in $G$ that is contained in some element of~$\cU$. Since metric spaces are paracompact, $H$ admits a locally finite refinement $\cV_H$ of $\cU_H$ consisting of basic open sets of $(H,\textsc{MTop})$.
Then the open cover $\cV$ of $|G|$ induced by $\cV_H$ gives the desired locally finite refinement of~$\cU$.

A similar argument shows that $|G|$ with \textsc{Top} is paracompact.
Here, $(H,\textsc{Top})$ is paracompact because all CW-complexes are.
\end{proof}

Note in particular that paracompactness implies normality and collectionwise normality, and hence we reobtain the previously mentioned results by Polat \cite[Lemma~4.14]{polat1996ends} and Sprüssel  \cite[Theorems~4.1 \& 4.2]{sprussel2008end} as a straightforward consequence of our Corollary~\ref{cor_modGpara}.

\section{Consequences of the approximation result}
\label{sec_cor}

In \cite[Theorem~5.13]{polat1996ends} Polat characterised the graphs that admit an end-faithful normal tree as the graphs with ultrametrizable end space, and raised the question \cite[\S10]{polat1996ends2} whether metrizability of the end space is enough to ensure the existence of an end-faithful normal tree. As our first application we show how using Theorem~\ref{thm_paracompactFast} provides a much simplified proof for Polat's result that simultaneously answers his question about the metrizable case in the affirmative:

\begin{theorem}\label{cor_char. OmegaG metr.}
For every connected graph $G$, the following are equivalent:
\begin{enumerate}
    \item The end space of $G$ is metrizable,
    \item the end space of $G$ is ultrametrizable, 
    \item $G$ contains an end-faithful normal tree.
\end{enumerate}
\end{theorem}
\begin{proof}
The implication (iii) $\Rightarrow$ (ii) is routine, as the end space of any tree is ultrametrizable (see e.g.\ \cite{hughes2004trees} for a detailed account), and $\Omega(T)$ and $\Omega(G)$ are homeomorphic for every end-faithful normal tree $T$ of $G$ (see e.g.\ \cite[Proposition~5.5]{diestel1992end}). The implication (ii) $\Rightarrow$ (i) is trivial.

Hence, it remains to prove (i) $\Rightarrow$ (iii). For this, consider the covers $\mathcal{U}_n$ for $n \in \N$ of $\Omega(G)$ given by the open balls with radius $1/n$ around every end; with respect to some fixed metric $d$ inducing the topology of $\Omega(G)$. By applying  Corollary~\ref{cor_ultrapara} to the covers $\cU_1,\cU_2,\ldots$, and combining it with Jung's Theorem~\ref{thm_jung}, it is straightforward to  
construct a sequence of  rayless normal trees $T_1 \subseteq T_2 \subseteq \ldots$ all rooted at the same vertex such that the partition of $\Omega(G)$ given by the components of $G - T_n$ refines~$\mathcal{U}_n$. 

Any two ends $\omega \neq \eta$ of $G$ are separated by any $T_n$ with $2/n < d(\omega, \eta)$. Consider the normal tree $T^\prime = \bigcup_{n \in \mathbb{N}} T_n$. We claim that each end $\omega \in \Omega(G) \setminus \partial_\Omega T^\prime$ belongs to a component $C$ of $G - T^\prime$ such that $N(C)$ is finite. Otherwise $N(C)$ lies on a unique normal ray $R$ of $T$ belonging to some end $\eta \in \partial_\Omega T^\prime$, but then clearly, none of the $T_n$ would separate $\omega$ from $\eta$, a contradiction. Hence, $N(C)$ is finite, and since $C$ contains at most one end, $T^\prime$  extends to an end-faithful normal tree of $G$. 
\end{proof}

From the new implication (i) $\Rightarrow$ (iii) in Theorem~\ref{cor_char. OmegaG metr.} one also obtains a simple proof of Diestel's characterisation from \cite{diestel2006end} when $|G|$ is metrizable.

\begin{corollary}\label{cor_char. modG metr.}
For every connected graph $G$, the following are equivalent:
\begin{enumerate}
    \item $\vert G \vert$ with \textsc{MTop} is metrizable,
    \item the space $\hat{V}(G)$ is metrizable, 
    \item $G$ has a normal spanning tree.
\end{enumerate}
\end{corollary}
\begin{proof}
The first implication (iii) $\Rightarrow$ (i) is routine, see e.g.\ \cite{diestel2006end}. The implication (i) $\Rightarrow$ (ii) is trivial. For (ii) $\Rightarrow$ (iii) apply Theorem~\ref{cor_char. OmegaG metr.} to the space $\Omega(G^+) \cong \hat{V}(G)$, noting that every end-faithful normal tree of $G^+$ is automatically spanning.
\end{proof}
To motivate our next applications, suppose that a given graph $G$ admits a normal spanning tree. Let us call such graphs \emph{normally spanned}. If $G$ is normally spanned, then every component of $G-X$ is normally spanned, too, for any finite $X \subseteq V(G)$.  Conversely, the question arises whether a graph admits a normal spanning tree as soon as every end $\omega$ has some basic neighbourhood $C(X,\omega)$ that is normally spanned. It turns out that the answer is yes:

\begin{theorem}
\label{thm_local NST gives global NST}
If every end of a connected graph $G$ has a normally spanned neighbourhood, then $G$ itself is normally spanned.
\end{theorem}
\begin{proof}
Let $\cC=\{\,C(X_\omega,\omega) \colon \omega \in \Omega(G)\,\}$ be a selection of normally spanned  neighbourhoods for all ends of $G$, and apply Theorem~\ref{thm_paracompactintro} to $\cC$ to find a rayless normal tree $T$ such that the collection of components of $G - T$ refines $\cC$. By Jung's Theorem~\ref{thm_jung}, each such component $C$ of $G-T$ is the union of countably many dispersed sets, say $V(C) = \bigcup_{n\geq 1} V^C_n$. But then $V_0 = V(T)$ together with all the sets $V_n:= \bigcup \{ \, V_n^C \colon C \text{ a component of } G - T \, \}$, for $n \geq 1$, witnesses that $V(G)$ is a countable union of dispersed sets. Hence, $G$ has a normal spanning tree by Jung's theorem.
\end{proof}

There is also a more topological viewpoint on the above result: The assumptions of Theorem~\ref{thm_local NST gives global NST} are by Corollary~\ref{cor_char. modG metr.} equivalent to the assertion that $\hat{V}(G)$ is locally metrizable. But locally metrizable paracompact spaces are metrizable, \cite[Exercise~5.4.A]{EngelkingBook}. Hence, applying Corollary~\ref{cor_char. modG metr.} once again to $\hat{V}(G)$ yields the desired normal spanning tree of $G$.

Continuing along these lines, we now address the question whether the existence of some \emph{local end-faithful normal tree} for every end of a graph already ensures the existence of an end-faithful normal tree of the entire graph. For a graph $G$ and an end $\omega$, we say that $\omega$ has a \emph{local end-faithful normal tree} if there is a normal tree $T$ in $G$ such that $\Abs{T}$ is a neighbourhood of $\omega$ in $\Omega(G)$.

\begin{theorem}\label{thm_local normal tree gives global normal tree}
If every end of a connected graph $G$ has a local end-faithful normal tree, then $G$ has an end-faithful normal tree.
\end{theorem}
\begin{proof}
By Theorem~\ref{cor_char. OmegaG metr.} every end in $\Omega(G)$ has a metrizable neighbourhood. But (ultra-)paracompact spaces which are locally metrizable are metrizable,  \cite[Exercise~5.4.A]{EngelkingBook}. Consequently, we have by Corollary~\ref{cor_ultrapara} that $\Omega( G )$ is metrizable. Applying again Theorem~\ref{cor_char. OmegaG metr.} yields the desired end-faithful normal tree of $G$.
\end{proof}

\section{Paracompactness in subspaces of end spaces}
\label{sec_last}

We conclude this paper with an observation concerning the following fundamental problem on the structure of end spaces raised by Diestel in 1992 {\cite[Problem~5.1]{diestel1992end}}:

\begin{mainproblem}
\label{prob_main}
Which topological spaces can be represented as an end space $\Omega(G)$ for some graph $G$?
\end{mainproblem}

In Corollary~\ref{cor_ultrapara} we established that end spaces are always ultra-paracompact. In this section we show that also all subspaces of end spaces inherit the property of being ultra-paracompact, i.e.\ that end spaces are hereditarily ultra-paracompact. This significantly reduces the number of topological candidates for a solution of Problem~\ref{prob_main}, and for example shows that certain compact spaces cannot occur as end space, which  Corollary~\ref{cor_ultrapara} wouldn't do on its own.

It is known that paracompactness and ultra-paracompactness, along with a number of other properties which are not per se hereditary such as normality and collectionwise normality, have the property that they are inherited by \emph{all} subspaces as soon as they are inherited by all \emph{open} subspaces. For the easy proof in case of paracompactness see e.g.\ Dieudonn\'e's original paper \cite[p.~68]{dieudonne1944generalisation}. Hence, our assertion follows at once from Corollary~\ref{cor_ultrapara} given the following  observation:

\begin{lemma}
\label{lem_opensubspaces}
Open subsets of end spaces are again end spaces.
\end{lemma}

\begin{proof}
Let $G$ be any graph, and consider some open, non-empty set $\Gamma \subseteq \Omega(G)$.
Write $\Gamma^{\complement}$ for its complement in $\Omega(G)$. 
Using Zorn's lemma, pick a maximal collection $\mathcal{R}$ of disjoint rays all belonging to ends in $\Gamma^\complement$, and let $W$ be the union $\bigcup \,\{\,V(R)\colon R \in \mathcal{R}\,\}$ of their vertex sets. 
Note that $\Abs{W}\subset\Gamma^\complement$ because $\Gamma^\complement$ is closed.
We claim that $\Gamma$ is homeomorphic to the end space of the graph $G':=G-W$.

In order to find a homeomorphism $\varphi \colon \Omega(G') \to \Gamma$, note first that, due to the maximality of $\mathcal{R}$, every ray in $G'$ is (as a ray of $G$) contained in an end of $\Gamma$. Consequently, every end $\omega'$ of $G'$ is contained in a unique end $\omega$ of $\Gamma$ and we define $\varphi$ via this correspondence.

To see that $\varphi$ is surjective, consider an open neighbourhood $\Omega(X,\omega) \subseteq \Gamma$, for a given $\omega \in \Gamma$. Then $W$ has only finite intersection with $C(X,\omega)$, as only finitely many rays from $\mathcal{R}$ can intersect $C(X,\omega)$, but do not have a tail in $C(X,\omega)$. So we may assume that $C(X,\omega)$ is contained in $G'$, by extending $X$. Now, every ray of $\omega$ contained in $C(X,\omega)$ gives an end in $G'$ that is mapped to $\omega$.

To see that $\varphi$ is injective, suppose there are two rays $R_1 , R_2$ in $G'$ that are not equivalent in $G'$ but equivalent in $G$. Then, there are infinitely many pairwise disjoint $R_1$-$R_2$ paths in $G$ and all but finitely many of these paths hit $W$. Then the end $\omega$ of $G$ containing $R_1$ and $R_2$ is an end in $\Gamma$ which lies in the closure of~$\Gamma^\complement$, contradicting the fact that $\Gamma^\complement$ is closed.

Finally, let us show that $\varphi$ is continuous and open. For the continuity of $\varphi$ remember that for any open set $\Omega(X,\omega) \subseteq \Gamma$ we may assume that $C(X, \omega)$ is contained in $G'$. In particular the preimage of  $\Omega(X,\omega)$ is open in $G'$. 

For $\varphi$ being open, consider an open set $ \Omega(X,\omega') \subseteq \Omega(G')$. Now, $C(X, \omega') \subseteq G' - X$ might not be a component of $G-X$. However, the set of vertices in $C(X, \omega')$ having a neighbour in $W$ is dispersed. Again by extending $X$, we may assume that $C(X, \omega')$ is a component of $G-X$. Consequently, its image is open in $\Omega(G)$.
\end{proof}

\begin{corollary}
\label{cor_herpara}
All end spaces are hereditarily ultra-paracompact. \qed
\end{corollary}

Interestingly, a careful reading of Sprüssel's proof that spaces $|G|$ are normal from \cite{sprussel2008end} establishes that every end space $\Omega(G)$ is in fact \emph{completely normal}, i.e.\ that  subsets with $\overline{A}\cap B = \emptyset = A \cap \overline{B}$ can be separated by disjoint open sets -- a property which is equivalent to hereditary normality, see \cite[Theorem~2.1.7]{EngelkingBook}. In any case, also this stronger result of hereditary normality is implied by our paracompactness result in Corollary~\ref{cor_herpara}.

\bigskip
{\bf Future steps.} The results in this paper narrow down which topological spaces occur as end spaces. Still, a complete solution to Problem~\ref{prob_main} seems currently out of reach without significant new insights. We remark that similar to Theorem~\ref{cor_char. OmegaG metr.},  one can show that end spaces $\Omega(G)$ have the property that every  separable subspace $X \subset \Omega(G)$, i.e.\ every such $X$ with a countable dense subset, must be metrizable. We currently do not  know of an example of a hereditarily  ultra-paracompact space where  every separable subspace is metrizable that does not occur as the end space $\Omega(G)$ of  some graph $G$.

\bibliographystyle{plain}
\bibliography{reference}

\end{document}